\documentclass[12pt]{amsart}
\usepackage{epsfig,color}
\usepackage{blindtext}
\usepackage{hyperref}
\usepackage{graphicx}
\usepackage{enumitem} 
\usepackage{url}
\usepackage{amssymb}
\usepackage[pagewise]{lineno}

\theoremstyle{definition}

\newtheorem{theorem}{Theorem}[section]
\newtheorem{definition}[theorem]{Definition}
\newtheorem{conjecture}[theorem]{Conjecture}

\newtheorem{lemma}[theorem]{Lemma}
\newtheorem{remark}[theorem]{Remark}

\newtheorem{example}[theorem]{Example}

\newtheorem*{remark*}{Remark}

\numberwithin{equation}{section}

\newcommand{\dv}{\text{div}}

\newcommand{\mc}{\mathcal}
\newcommand{\mb}{\mathbb}

\newcommand{\eps}{\varepsilon}

\newcommand{\bn}{\mathbf n}

\DeclareMathOperator{\Ric}{Ric}
\DeclareMathOperator{\Hess}{Hess}

\DeclareMathOperator{\dist}{dist}

\title[Compactness of Constant Mean Curvature Surfaces]{Compactness of Constant Mean Curvature Surfaces in Three Manifold with Positive Ricci Curvature}

\date{\today}

\author{Ao Sun}
\address{Department of Mathematics, Massachusetts Institute of Technology, Cambridge, MA 02139, USA}
\email{aosun@mit.edu}

\begin{document}
\maketitle
\begin{abstract}
In this paper we prove a compactness theorem for constant mean curvature surfaces with area and genus bound in three manifold with positive Ricci curvature. As an application, we give a lower bound of the first eigenvalue of constant mean curvature surfaces in three manifold with positive Ricci curvature.
\end{abstract}

\section{Introduction}
Let $M$ be a $3$ dimensional manifold and $\Sigma\subset M$ be a surface. Let $H$ be the mean curvature of $\Sigma$. We say $\Sigma$ is a constant mean curvature (CMC) surface if $H$ is a constant. In particular, if $H$ is constant $0$, $\Sigma$ is a minimal surface. There are many examples of CMC surfaces in $\mb R^3$, see \cite{meeks2016constant}. Recently in \cite{zhou2017min} Zhou-Zhu proved the existence of embedded CMC hypersurfaces in closed $n$ dimensional manifold with $2\leq n\leq 6$.

In this paper we prove the following compactness theorem for embedded CMC surfaces in three manifold with positive Ricci curvature:

\begin{theorem}\label{T:Main theorem}
Let $M$ be a $3$ dimensional compact manifold with positive Ricci curvature and no boundary. Suppose $\Sigma_i\subset M$ is a sequence of closed embedded CMC surfaces with constant mean curvature $H_i$, satisfies the following conditions:
\begin{enumerate}
\item $\vert H_i\vert\leq H_0$ for some constant $H_0$,
\item The genus of $\Sigma_i$ are uniformly bounded,
\item The area of $\Sigma_i$ are uniformly bounded.
\end{enumerate}
Then 

\begin{enumerate}
\item either: there is an self-touching smoothly immersed CMC surface $\Sigma$ with finitely many neck pinching points, such that a subsequence of $\Sigma_i$ converges to $\Sigma$ in $C^k$ topology for any $k\geq2$ besides those neck pinching points.
\item or: there is a embedded minimal surface $\Sigma$ such that $\Sigma_i$ converges to $\Sigma$ with multiplicity $2$.
\end{enumerate}
\end{theorem}

Here we say $\Sigma$ is {\bf self-touching} if at any non-embedded point $p\in\Sigma$, there is a small $r$ such that $B_r(p)\cap\Sigma$ is a union of two disks $D_1,D_2$, and $D_1$ can be written as a graph of function $\phi$ over $D_2$ where $\phi\geq0$ on $D_2$. Intuitively this means that $\Sigma$ is immersed but can not across itself. {\bf Neck pinching} points are special touching points. We will give precise definition in section \ref{section:Compactness Theorem}. Intuitively one can image we pinching a piece of plasticine into two pieces, and at the moment they are just detached, the point connects them is a neck pinching point.

\subsection{Compactness Theorem}

The compactness theorem for minimal surfaces was first developed by Choi-Schoen in \cite{choi1985space}. They proved the compactness theorem of embedded minimal surfaces in $3$ dimensional manifold with positive Ricci curvature. Later their result was generalized into many other situations. For example, in \cite{white1987curvature} White generalized the compactness theorem for surfaces which are stationary for parametric elliptic functionals, in \cite{colding2012smooth} Colding-Minicozzi generalized the compactness theorem to self-shrinkers in $\mb R^3$. We will follow the key ideas of these papers.

There are two main ingredients in Choi-Schoen's proof. The first ingredient is curvature estimate for minimal surface. Then we can get some uniform curvature bound besides finitely many points, so we can find a subsequence of $\Sigma_i$ converges smoothly to a limit surface $\Sigma$ besides finitely many points. Here we need to generalize the curvature estimate to CMC surfaces, and get an uniform curvature estimate only depends on the uniform mean curvature bound $H_0$.

The second ingredient is showing the multiplicity of the convergence is one. Then by a result of Allard in \cite{allard1972first} the convergence is smooth. There are two methods to show the multiplicity is one. Choi-Schoen argued that if the convergence has multiplicity more than one, then they can construct a family of functions which contradict the eigenvalue estimate in Choi-Wang; another method by White and Colding-Minicozzi argued that if the mulipicity is more than one, then the linearized operator may have positive Jacobi field, which is impossible if $M$ has positive Ricci curvature. We will follow the second method, because now we do not have eigenvalue estimate for CMC surfaces. In our case, a key observation is that although CMC surfaces and minimal surfaces satisfy different equations, their linearized operators are the same. Hence we may conduct the same argument as the minimal surfaces case.

If the limit surface is minimal, then the CMC surfaces may approach it on both sides with different orientation, and the differential operator is not the same as the differential operator of minimal surfaces. Thus, the convergence may not be multiplicity $1$. However, if the multiplicity of the convergence is more than $2$, then we can find two sheets with the same orientation, then the differential operator is again the same as the differential operator of minimal surfaces. Then we can get the a positive Jacobi field again to argue for contradiction.

As an application of the compactness Theorem \ref{T:Main theorem}, we can actually get a lower bound for the first eigenvalue of CMC surfaces. 

\begin{theorem}\label{T:Application eigenvalue}
Let $M$ be a three dimensional manifold with positive Ricci curvature. Suppose there is no embedded minimal surface in $M$ which is the multiplicity $2$ limit of a sequence of CMC surfaces. Then for any embedded CMC surface with area bound $V$, genus bound $G$ and mean curvature bound $\vert H\vert\leq H_0$, we have the first eigenvalue lower bound:

\begin{equation}
\lambda\geq \frac{\min\Ric -HC}{2},
\end{equation}
where $C$ is a constant depending on $M,V,G,H_0$.
\end{theorem}

\subsection{Touching Phenomenon}

Touching phenomenon would not appear in minimal surfaces due to maximum principle, but it may appear in CMC surfaces, especially in a convergence process. The appearance of touching points makes the convergence of CMC surfaces much more complicated. 

Touching phenomenon is natural in our physical world. For example, when taking a shower one can observe many soap bubbles touching each other. More complicated examples appear in general three manifold rather than $\mb R^3$, and we give some examples in section \ref{section:Touching}.

In general touching points in the limit would not influence the smooth convergence in our main Theorem \ref{T:Main theorem} if they are generated when two part of the surface kissing each other. One can image the convergence is smooth on each pieces. See the first picture:
\begin{figure}\label{figure1}
\centering
\scalebox{1.25}{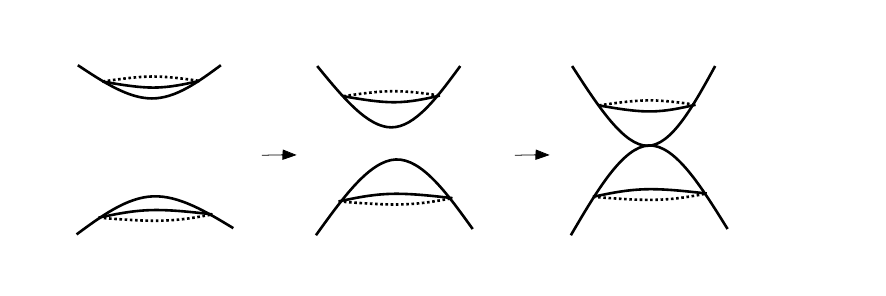}
\end{figure}

However neck pinching points are generated with some topological changes, so smooth convergence can not across these points. See the second picture.  
\begin{figure}\label{figure2}
\centering
\scalebox{1.25}{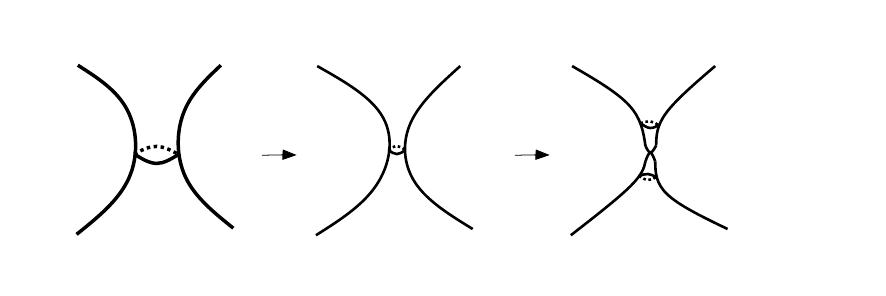}
\end{figure}   

Neck pinching phenomenon is very common in geometric analysis. For example, neck pinching phenomenon appear in geometric flows, such as mean curvature flow (see \cite{gang2009neck}) and Ricci flow (see \cite{angenent2004example}). In order to deal with the non-smoothness of the flow across these neck pinching points, one need to do surgery for the geometric flows. For example in \cite{perelman2002entropy} Perelman studied surgery of Ricci flows and in \cite{brendle2016mean} Brendle and Huisken studied surgery of mean curvature flows in $\mb R^3$.

\subsection{Organization of the Paper}
In section 2 we will prove the Choi-Schoen type curvature estimate for CMC surfaces. We will just follow Choi-Schoen's proof; A similar estimate for CMC surfaces in $\mb R^3$ appears in \cite{zhang2004curvature}.

In section 3 we discuss the linearized equation and the linearized operator. 

In section 4 we prove the main compactness theorem. We follow the idea by White \cite{white1987curvature} and Colding-Minicozzi \cite{colding2012smooth}.

In section 5 we present some touching examples of CMC surfaces. 

In section 6 as an application of the compactness theorem, we prove a lower bound for the first eigenvalue of CMC surfaces.

\subsection{Acknowledgement}
The author want to thank professor William Minicozzi and professor Xin Zhou for helpful discussions and comments. We also want to thank Jonathan Zhu for pointing out the possible multiplicity $2$ convergence when the limit is minimal.

\section{Curvature Estimate}
In this section we generalize the Choi-Schoen curvature estimate for minimal surfaces to CMC surfaces. We first need some tools.

\subsection{Tools for Curvature Estimate}
The first lemma is a Simon's type inequality for CMC surfaces. We need to keep track of the mean curvature term.

\begin{lemma}\label{lemma-CMC Simons inequality}
Suppose $\Sigma$ is a CMC surfaces with mean curvature $H$, $\vert H\vert\leq H_0$. Then 
\begin{equation}
\Delta_\Sigma\vert A\vert^2\geq -C(\delta^2+\vert A\vert^2)^2.
\end{equation}
where $C$ is a universal constant, and $\delta$ is quadratic under the rescaling of the $M$, i.e. suppose we rescale metric $g$ to $\tilde g=\sigma g$, then $\delta$ becomes $\tilde\delta=\sigma^{-1}\delta$.
\end{lemma}

\begin{proof}
See \cite{ilias2012caccioppoli} Theorem 3.1. 
\end{proof}

Next lemma generalizes the monotonicity formula for minimal surfaces to CMC surfaces.

\begin{lemma}\label{lemma-CMC monotonicity}
Suppose $M$ is a closed three manifold with sectional curvature bounded by $k$ and injective radius bounded from below by $i_0$. Suppose $\Sigma\subset M$ is a CMC surfaces with mean curvature $H$, $\vert H\vert\leq H_0$. Suppose $f$ is a function on $\Sigma$ satisfies $\Delta_\Sigma f\geq -\lambda t^{-2}f$, where $\lambda$ is a fixed constant and $t<\min\{i_0,1/\sqrt{k}\}$. Then we have
\begin{equation}
f(x_0)\leq \frac{e^{\lambda+C(H_0,k) t/2}}{\pi}\int_{B_t(x_0)\cap\Sigma}f.
\end{equation}
\end{lemma}

Before we prove this lemma, let us state a lemma of the famous Laplacian comparison theorem in three manifold. See \cite{colding2011course} Chapter 7 Lemma 7.1 for proof.

\begin{lemma}\label{lemma-Laplacian comparison}
Suppose $M$ is a closed three manifold with sectional curvature bounded by $k$ and injective radius bounded from below by $i_0$. Let $x\in M$ be a fixed point, and $r$ is the distance function from $x$. Then for $r<\min\{i_0,1/\sqrt{k}\}$ and any vector $X$ with $\vert X\vert=1$,
\begin{equation}
\vert\Hess_r(X,X)-\frac{1}{r}\langle X-\langle X,Dr\rangle Dr,X-\langle X,Dr\rangle Dr\rangle\vert\leq\sqrt{k}.
\end{equation}
Here $D$ is the gradient on $M$.
\end{lemma}

\begin{proof}[proof of Lemma \ref{lemma-CMC monotonicity}]
Let $y\in\Sigma$ be a point such that $r(y)<\min\{i_0,1/\sqrt{k}\}$. We choose a local orthonormal frame $E_1,E_2$. Then by Laplacian comparison Lemma \ref{lemma-Laplacian comparison}, we have
\begin{equation}
\vert\Hess_r(E_1,E_1)-\frac{1}{r}\langle E_1-\langle E_1,Dr\rangle Dr,E_1-\langle E_1,Dr\rangle Dr\rangle\vert\leq\sqrt{k},
\end{equation}
\begin{equation}
\vert\Hess_r(E_2,E_2)-\frac{1}{r}\langle E_2-\langle E_2,Dr\rangle Dr,E_2-\langle E_2,Dr\rangle Dr\rangle\vert\leq\sqrt{k}.
\end{equation}
Add these two inequalities and note $\Sigma$ is a CMC surface, we get (Compare to minimal surfaces case in Colding-Minicozzi \cite{colding2011course} Chapter 7 (7.2))
\begin{equation}
\vert\Delta_\Sigma r^2-4-\langle\nabla^\bot r^2,H\mathbf{n}\rangle\vert\leq 4\sqrt{k}r.
\end{equation}
Note $\vert Dr\vert\leq1$, we get
\begin{equation}\label{equation-Hessian comparison}
\vert\Delta_\Sigma r^2-4\vert\leq (4\sqrt{k}+2H_0)r=\alpha r.
\end{equation}
where $C$ only depends on $k,H_0$.
Let us define 
\[F(s)=\frac{1}{s^2}\int_{B_s(x_0)\cap\Sigma}f.\]
We can differentiate it for almost every $s<t$ 
\begin{equation}
F'(s)=-\frac{2}{s^3}\int_{B_s(x_0)\cap\Sigma}f+\frac{1}{s^2}\int_{\partial B_s(x_0)\cap\Sigma}\frac{f}{\vert\nabla_\Sigma r\vert}.
\end{equation}
Here we use the co-area formula, see \cite{colding2011course} page 24 (1.59). Let us estimate the first term on the right hand side. Using inequality (\ref{equation-Hessian comparison}) and integrating by parts gives
\begin{equation}
\begin{split}
-\frac{2}{s^3}\int_{B_s(x_0)\cap\Sigma}f&=-\frac{1}{2s^3}\int_{B_s(x_0)\cap\Sigma}4f\\
&\geq -\frac{1}{2s^3}\int_{B_s(x_0)\cap\Sigma}\Delta_\Sigma(r^2-s^2)f-\frac{1}{2s^3}\int_{B_s(x_0)\cap\Sigma}\alpha rf\\
&=-\frac{1}{2s^3}\int_{B_s(x_0)\cap\Sigma}(r^2-s^2)\Delta_\Sigma f-\frac{1}{2s^3}\int_{\partial B_s(x_0)\cap\Sigma}\nabla_\Sigma(r^2)f-\frac{1}{2s^3}\int_{B_s(x_0)\cap\Sigma}\alpha rf.
\end{split}
\end{equation}

So we get the following inequality
\begin{equation}
\begin{split}
F'(s)&\geq \frac{1}{2s^3}\int_{B_s(x_0)\cap\Sigma}(s^2-r^2)\Delta_\Sigma f+\frac{1}{s^2}\int_{\partial B_s(x_0)\cap\Sigma} \frac{1-\vert\nabla_\Sigma r\vert^2}{\vert\nabla_\Sigma r\vert}f-\frac{\alpha}{2} F(s)\\
&\geq \frac{1}{2s^3}\int_{B_s(x_0)\cap\Sigma}(s^2-r^2)\Delta_\Sigma f-\frac{\alpha}{2}F(s)\\
&\geq -\frac{1}{2s^3}\int_{B_s(x_0)\cap\Sigma}(s^2-r^2)t^{-2}\lambda f-\frac{\alpha}{2}F(s)\\
&\geq -\frac{\lambda}{t}F(s)-\frac{\alpha}{2}F(s).
\end{split}
\end{equation}
In conclusion, $e^{(\frac{\lambda}{t}+\frac{\alpha}{2})s}F(s)$ is monotone non-decreasing. Then we can conclude that
\begin{equation}
f(x_0)\leq \frac{e^{C+\alpha t/2}}{\pi}\int_{B_t(x_0)\cap\Sigma}f.
\end{equation}

\end{proof}

\subsection{Choi-Schoen Type Estimate}
\begin{theorem}\label{theorem-Choi-Schoen}
Let $M$ be a $3$ dimensional manifold. Let $p\in M$ and $r>0$ such that $B_r(p)$ has compact closure in $M$. Suppose $\Sigma$ is a compact immersed CMC surface with mean curvature $H$ in $M$ such that $B_r(p)\cap\partial \Sigma=\emptyset$. Here $\vert H\vert\leq H_0$. Then there exists $\eps_0>0$ depending on the geometry of $B_r(p)$ and $H_0$ such that if 
\[\int_{\Sigma\cap B_r(p)}\vert A\vert^2\leq\eps_0.\]
and $r\leq\eps_0$, then
\begin{equation}
\max_{0\leq\sigma\leq r}\sigma^2\sup_{B_{r-\sigma}(p)}\vert A\vert^2\leq C=C(H_0,B_{r}(p)).
\end{equation}
\end{theorem}

\begin{proof}
We follow the idea of Choi-Schoen. Since $\sigma^2\sup_{B_{r-\sigma}(p)}\vert A\vert^2$ vanishes on $\partial B_r$, the supremum of $\sigma^2\sup_{B_{r-\sigma}(p)}\vert A\vert^2$ must be achieved inside $B_r$. Let $\sigma_0$ be the number so that 
\[\sigma_0^2\sup_{B_{r-\sigma_0}(p)}\vert A\vert^2=\max_{0\leq\sigma\leq r}\sigma^2\sup_{B_{r-\sigma}(p)}\vert A\vert^2.\]
and let $q\in B_{r-\sigma_0}(p)$ be chosen to that 
\[\vert A\vert^2(q)=\sup_{B_{r-\sigma_0}(q)}\vert A\vert^2.\]
Then
\begin{equation}\label{equation-bounded A in choi schoen}
\sup_{B_{\frac{1}{2}\sigma_0}(q)}\vert A\vert^2\leq 4\vert A\vert^2(q).
\end{equation}
If $\sigma_0^2\vert A\vert^2(q)\leq 4$ then the inequality holds. So we only need to consider the case $\sigma_0^2\vert A\vert^2(q)> 4$. Then we rescale the metric $ds^2$ on $M$ by setting $\tilde ds^2=\vert A\vert^2(q)ds^2$, and we denote the balls and the quantities under rescaled metric with tilde. $\sigma_0^2\vert A\vert^2(q)> 4$ implies that $\partial \Sigma\cap\tilde B_1(q)=\emptyset$. Inequality (\ref{equation-bounded A in choi schoen}) implies
\[\sup_{\tilde B_1(q)}\vert\tilde A\vert^2\leq 4.\]

By Simons' inequality Lemma \ref{lemma-CMC Simons inequality}

\[\tilde\Delta_\Sigma\vert \tilde A\vert^2\geq -C(\delta^2+\vert A\vert^2)^2.\]

Note here $\delta^2\leq C\sigma_0^2\leq C\eps_0^2$. Together with inequality $\sup_{\tilde B_1(q)}\vert\tilde A\vert^2\leq 4$ we get
\[\tilde\Delta_\Sigma u\geq -Cu\mbox{ on }\tilde B_1(q),\]
where $u=\delta^2+\vert \tilde A\vert^2$ and $C$ here is a universal constant. So monotonicity formula Lemma \ref{lemma-CMC monotonicity} gives
\begin{equation}
u(x_0)\leq \frac{e^{C+\tilde\alpha/2}}{\pi}\int_{\tilde B_1(q)\cap\Sigma}u.
\end{equation}
Note $\tilde\alpha\leq \alpha\sigma_0\leq\alpha\eps_0$
\begin{equation}
\vert A\vert^2(q)\leq u(q)\leq \frac{e^{C+\alpha\eps_0/2}}{\pi}\int_{\tilde B_1(q)\cap\Sigma}(\delta^2+\vert A\vert^2)\leq C\eps_0.
\end{equation}
where we use the conformal invariant of the integral of $\vert A\vert^2$ and the area bound of CMC surface. If $\eps_0$ small enough, we get a contradiction since $\vert\tilde A\vert^2(q)=1$. Thus we finish the proof.

\end{proof}

\section{Linearized Equation}
Let $\Sigma$ be a CMC surface in $M$. Let us define a differential operator $L$ such that
\begin{equation}
Lu=\Delta_\Sigma u+\Ric(\mathbf{n},\mathbf{n})u+\vert A\vert^2u.
\end{equation}
We call $L$ is the {\bf linearized operator}. In this section we study some properties of this operator. 
\subsection{Difference of Two CMC Surfaces}
Let $M$ be a $3$ dimensional manifold. Suppose $\Sigma_1,\Sigma_2\subset M$ are two constant mean curvature surfaces, with mean curvature $H_1,H_2$ respectively.

\begin{theorem}\label{difference}
Let $\Sigma_1,\Sigma_2$ be two CMC surfaces with constant mean curvature $H_1,H_2$ respectively. Suppose $\Sigma_2$ is a graph over $\Sigma_1$, i.e.
\[\Sigma_2=\{x+\varphi\mathbf{n}:x\in\Sigma_1\}.\]
Then $\varphi$ satisfies the second order elliptic equation

\begin{equation}
Lu-(H_2-H_1)=\dv(a\nabla\varphi)+b\cdot\nabla \varphi+c\varphi.
\end{equation}
Here $a,b,c$ turns to $0$ as $\Vert\varphi\Vert_{C^2}$ goes to $0$.
\end{theorem}
This can be viewed as a non-infinitesimal version of second variational formula. We left the computations in the appendix. Intuitively, one can image the second order variational formula gives the second order derivative of minimal surfaces, and the difference formula here gives the Taylor expansion of the minimal surfaces up to second order. In particular, if we let $\varphi\to 0$, we will again get the second variational formula.

\subsection{Stability of Linearized Operator}
Suppose $\Sigma$ is a CMC surface. We say the linearized operator $L$ of $\Sigma$ is {\bf stable} if for any function $u$ on $\Sigma$,

\begin{equation}
\int_{\Sigma}uLu\leq0.
\end{equation}

Otherwise we say $L$ is unstable. Note this definition of stability is not the same as the stability of the CMC surface itself, since when we talk about the stability of the CMC surface $\Sigma$, we only consider the variational fields which preserve the volume enclosed by $\Sigma$.

For $\Sigma$ be a surface in a three manifold $M$ with positive Ricci curvature, let $u\equiv 1$ we see that
\[\int_\Sigma uLu=\int_\Sigma \vert A\vert^2+\Ric(\mathbf{n,\mathbf{n}})>0.\]
Hence $L$ is always unstable.

Recall a Jacobi field on $\Sigma$ is a variational field $f\mathbf{n}$ such that $Lf=0$. The following lemma show that a positive Jacobi field implies the stability of $L$.

\begin{lemma}\label{lemma-positive Jacobi field}
Suppose there is a positive function $u$ on $\Sigma$ such that $Lu=0$. Then $L$ is stable.
\end{lemma}

\begin{proof}
Let $w=\log u$. Then $\Delta_\Sigma w=-\vert A\vert^2-\Ric(\mathbf{n},\mathbf{n})-\vert\nabla_\Sigma w\vert^2$.

Let $v$ be any smooth function on $\Sigma$. Multiplying the above identity by $v^2$ to get
\begin{equation}
\begin{split}
\int_\Sigma v^2(\vert A\vert^2+\Ric(\mathbf{n},\mathbf{n}))+\int_\Sigma \vert\nabla_\Sigma w\vert^2v^2&=-\int_\Sigma v^2\Delta_\Sigma w =2\int_\Sigma v\langle \nabla_\Sigma v,\nabla_\Sigma w\rangle\\
&\leq2\int_\Sigma\vert v\vert\vert\nabla_\Sigma w\vert\vert\nabla_\Sigma v\vert \leq \int_{\Sigma}v^2\vert\nabla_\Sigma w\vert^2+\int_\Sigma\vert\nabla_\Sigma v\vert^2.
\end{split}
\end{equation}
Then integration by parts gives
\begin{equation}
\int_\Sigma vLv\leq 0.
\end{equation}
\end{proof}
\section{Compactness Theorem}\label{section:Compactness Theorem}
In this section, we will prove the main compactness theorem.

\subsection{Smooth Limit}
We first show there is a reasonable smooth limit under the conditions in Theorem \ref{T:Main theorem}. 
\begin{theorem}\label{T:Choi-Schoen Convergence}
Let $M$ be a $3$ dimensional compact manifold with positive Ricci curvature and no boundary. Suppose $\Sigma_i\subset M$ is a sequence of closed embedded CMC surfaces with constant mean curvature $H_i$, satisfies the following conditions:
\begin{enumerate}
\item $\vert H_i\vert\leq H_0$ for some constant $H_0$,
\item The genus of $\Sigma_i$ are uniformly bounded,
\item The area of $\Sigma_i$ are uniformly bounded.
\end{enumerate}
Then there is a self-touching smoothly immersed CMC surface $\Sigma$ such that a subsequence of $\Sigma_i$ converges to $\Sigma$ in $C^k$ topology for any $k\geq2$ besides a finite singular set $\mc S$.
\end{theorem}

\begin{proof}
We follow \cite{choi1985space} and \cite{colding2012smooth}. For each positive integer $m$, take a finite covering $\{B_{r_m}(y_j)\}$ of $M$ such that each point of $M$ is covered at most $h$ times by balls in this covering, and $\{B_{r_m/2}(y_j)\}$ is still a covering of $M$. Here we set $r_m=2^{-m}\eps_0$ and $h$ only depends on $M$. Then we have
\[\sum_{j}\int_{\Sigma_i\cap B_{r_m}(y_j)}\vert A\vert^2\leq hC\]
Therefore for each $i$ at most $hC/\eps_0$ number of balls on which
\[\int_{\Sigma_i\cap B_{r_m}(y_j)}\vert A\vert^2\geq\eps_0\]
By passing to a subsequence of $\Sigma_i$ we can always assume that all the $\Sigma_i$ has the same balls with total curvature $\geq\eps_0$. Call the center of these balls to be $\{x_{1,m},\cdots,x_{l,m}\}$, where $l$ is an integer at most $hC/\eps_0$. Then on the balls other than $B_{x_{k,m}(r_m)}$, by Theorem \ref{theorem-Choi-Schoen} we have uniformly point-wise curvature bound for $\Sigma_i$. Passing to subsequence we may assume that $\Sigma_i$'s converges smoothly on a half of those balls to $\Sigma$. Since $\Sigma_i$ are embedded, so the limit $\Sigma$ is self-touching in the balls other than $B_{x_{k,m}(r_m)}$.

Then we can continue this process as $m$ increase. Finally by a diagonal argument we can get a subsequence $\{\Sigma_i\}$, converges smoothly everywhere to $\Sigma$ besides those points $x_1,\cdots,x_l$ which is the limit of those $\{x_{1,m}\},\cdots,\{x_{l,m}\}$. Moreover, since there is no maximum principle for $\lambda$-surfaces, the limit is only immersed. However if we consider the compactness for each connected components in any fixed ball, we can see the limit is self-touching away from $x_1,\cdots,x_l$.
\end{proof}

Next we will show that $\Sigma$ is actually smooth everywhere. We will follow White to prove that the singularities are removable. The main ingredient is a more delicate curvature estimate near the singularities.

\begin{lemma}\label{curvaturenearsingularity}
Suppose $\Sigma$ is a properly self-touching CMC surface in $B_R\setminus\{x_0\}$ with mean curvature $\vert H\vert\leq H_0$, then there exists $\eps=\eps(H_0,R,x_0)>0$ such that if $\int_{\Sigma}\vert A\vert^2\leq\eps$, then there is $C$ such that
\begin{equation}
\vert A(x)\vert\vert \dist(x,x_0)\vert\leq C.
\end{equation}
\end{lemma}

\begin{proof}
We show by contradiction. If the criterion is not true, then we can find a sequence of points $x_n\in((B_{R}(x_0)\setminus B_{1/n}(x_0))\cap\Sigma)$ such that
\[\vert A(x_n)\vert^2(\dist(x,x_0)-\frac{1}{n})\to+\infty.\]
 Otherwise we will have uniform bound for $\vert A(x)\vert^2(\dist(x,x_0)-\frac{1}{n})$ for a sequence of $n\to\infty$, then passing to limit we will have a uniform bound for $\vert A(x)\vert^2\dist(x,x_0)$.

Then we can choose $z_n\in ((B_{R}(x_0)\setminus B_{1/n}(x_0))\cap\Sigma)$ such that $\vert A(z_n)\vert^2(\dist(z_n,x_0)-\frac{1}{n})$ achieve maximum. Note that $\vert A(x)\vert^2(\dist(x,x_0)-\frac{1}{n})$ achieve $0$ on $\partial B_{1/n}(x_0)\cap\Sigma$, so $d_n:=\dist(z_n,x_0)-\frac{1}{n}>0$.

Now we rescale $B_{d_n/2}(z_n)$ with $\vert A(z_n)\vert$, and denote $\{x\in\Sigma:\dist(x,z_n)\leq d_n/2\}$ after rescaling by $\tilde\Sigma_n$, and use tilde to denote the quantities on this new surface. Moreover, since $\vert A(z_n)\vert \to\infty$, the limit of the rescaling of $B_{d_n/2}(z_n)$ would converges to $\mb R^3$, so we can assume $n$ large such that $\tilde\Sigma_n$ actually live in $\mb R^3$ with slightly perturbation of standard Euclidean metric.

Note $\tilde\Sigma_n$ satisfy the following properties. First, $\vert \tilde A(0)\vert=1$; Second, by 
\[\vert A(z_n)\vert^2d_n\to+\infty\]
we know that for any fixed $R>0$, $\tilde\Sigma_n\cap \partial B_{R}(0)\not=\emptyset$ in $\mb R^3$ if $n$ large enough, and $\partial\tilde\Sigma_n\cap B_R(0)=\emptyset$ if $n$ large enough; Finally, for any $x'=\vert A(z)\vert x\in\tilde\Sigma_n$, we have

\[\vert A(x)\vert(\dist(x-x_0)\dist-\frac{1}{n})\leq \vert A(z)\vert d_n.\]

Since $\dist(x,z)\leq d_n/2$, we have $\dist x,x_0)-\frac{1}{n}\geq d_n/2$, thus $\vert A(x)\vert\leq 2\vert A(z)\vert$, thus $\vert \tilde A(x')\vert\leq 2$.
 
By the uniform curvature bound for $\tilde\Sigma_n$, for each $R>0$, there exists a subsequence (still denoted by $\tilde\Sigma_n$) converging smoothly on $B_R(0)$ to a complete surface $\tilde\Sigma$. Checking the equation of rescaling, we see that the limit $\tilde\Sigma$ must be a minimal surface, i.e. $\tilde H=0$.

Since the rescaling would not change the integral of the squared curvature, we have

\[\int_{B_R(0)\cap\tilde\Sigma}\vert A\vert^2\leq\eps.\]

Thus $\tilde\Sigma$ has to be the plane. Which contradicts to the condition that $\vert \tilde A(0)\vert=1$.

\end{proof}

\begin{theorem}\label{limit}
The limit surface in Theorem \ref{T:Choi-Schoen Convergence} is smoothly immersed. Moreover, for $y\in \mc S$ be a non-embedded point, in a small neighbourhood of $y$, $\Sigma$ is a union of two disks which are touching at $y$.
\end{theorem}

\begin{proof}
We only need to prove $\Sigma$ is smooth around singular set $\mc S$. Suppose $y\in\mc S$ is a singularity. We may assume $r$ small enough such that $\int_{B_r(y)\cap\Sigma}\vert A\vert^2\leq\eps$ (Note $\Sigma$ has finite total curvature since $\Sigma_i$'s have uniform finite total curvature). 

By Lemma \ref{curvaturenearsingularity}, there is a constant $C$ such that
\[\vert A(x)\vert\dist(x,y)\leq C.\]
For any $x\in B_r(y)\cap\Sigma$. Now we choose a sequence $r_i\to 0$ and rescale $B_r(y)$ and $\Sigma_i$ by $1/r_i$ and denote it by $\tilde\Sigma_i$. Note the curvature bound
\[\vert A(x)\vert\dist(x,y)\leq C\]
is invariant under rescaling, so this uniform curvature bound indicate that $\tilde\Sigma_i$ smoothly converges to a complete surface $\tilde\Sigma$ in $\mb R^3\setminus\{0\}$. See \cite{white1987curvature}.

Now for $K$ be any compact subset of $\mb R^3\setminus\{0\}$, 
\[\int_{\tilde\Sigma_i\cap K}\vert A\vert^2=\int_{\Sigma_i\cap r_iK}\vert A\vert^2\to0\text{ as $r_i\to0$}.\]
This implies $\tilde\Sigma$ is a union of planes. Thus $\Sigma\cap B_r(0)$ is actually a union of disks and punctured disks.

Now let $\Sigma$ denote one of its connected components which is a punctured disk. Since $\tilde\Sigma_i$ converges to to the plane in $\mb R^3\setminus\{0\}$, we can assume for some $i$, $\tilde\Sigma_i$ can be written as a graph $\varphi_i$ of that plane. Without lost of generality, let the plane be the $xy$ plane in $\mb R^3$. By the computations in the appendix, in $B_1$, $\tilde\Sigma_i$ satisfies a elliptic equation over the tangent plane:
\begin{equation}
L\varphi-(H_2-H_1)=\dv(a\nabla\varphi)+b\cdot\nabla \varphi+c\varphi.
\end{equation}
Here every terms are defined on $\mb R^2\cap B_1(0)$. Again, when $i$ large, each terms on the right hand side goes to $0$. Then by implicit function theorem, if we fixed the normal direction to point upwards, we can solve $\varphi_{i,t}$ for boundary date $\varphi_{i,t}=\varphi_i+t$ on $\partial B_{1}(0)$. Then the graphs of $\varphi_{i,t}$ foliate a region of $B_1(0)\times\mb R^2$. Since we fixed the direction of normal vectors, we can apply maximal principle, which indicates that the leaf such that $\varphi_{i,t}(0)=0$ lies on one side of $\tilde\Sigma_i$. As a result, any sequence of dilations of $\Sigma$ must converge to the same limiting plane, which is just the tangent plane of that leaf at $0$.

Thus $\Sigma\cup\{0\}$ is a $C^1$ graph of a function $v$ in a neighbourhood of $0$. Since $v$ is a $C^{2,\alpha}$ solution to an elliptic equation except $0$, then $v$ is actually $C^{2,\alpha}$. Hence $\Sigma\cup\{0\}$ is a smooth disk.

We have already shown that $\Sigma\cup\{0\}$ is a union of smooth disks. So $\Sigma$ is an immersed surface, with locally finite many curvature concentration points. By maximal principle, at each touching point $\Sigma$ consists of two disks which are touching at that point. So $\Sigma$ is smoothly self-touching immersed.
\end{proof}

\subsection{Smoothly Convergence}
In this subsection we first assume that the limit is not minimal, and discuss the situation that the limit is minimal in the end. 

We will show the convergence is smooth besides neck pinching points. Note we have already shown the convergence is smooth besides $\mc S$, so we only need to show smooth convergence across points in $\mc S$ with density $1$ (i.e. locally $\Sigma$ is one disk) and points in $\mc S$ which is not neck pinching. 

We first define neck pinching points. From Theorem \ref{limit} we know that for any points $y\in \mc S$ with density more than one, locally $\Sigma$ is the union of two disks $D_1$, $D_2$, and $\Sigma_i$ can be written as graphs $G_i^1,G_i^2$ of functions $\varphi_i^1$, $\varphi_i^2$ over $D_1\setminus\{y\}$ and $D_2\setminus\{y\}$ respectively. 
\begin{definition}\label{Definition:neck pinching point}
We say $y$ is a {\bf neck pinching} point if there is $r_0>0$ such that for $0<r<r_0$, $G_i^1$ and $G_i^2$ do not lie in the same connected components of $\Sigma_i\cap B_r$ for at most finitely many $\Sigma_i$'s.
\end{definition}

Now we can prove the convergence is smooth besides these neck pinching points. The main ingredient is to show the convergence has multiplicity one. Then by a theorem by Allard \cite{allard1972first}, also see \cite{choi1985space} and \cite{colding2012smooth}, we can show the convergence is smooth across those singularities with density $1$. Finally we show that even for a singularity with density greater than $1$, if it is not a neck pinching point we can still argue that the convergence is smooth across it.

\begin{theorem}\label{T:multiplicity1}
The multiplicity of the convergence in Theorem \ref{T:Choi-Schoen Convergence} is one when the limit is not minimal. 
\end{theorem}

We follow the idea in \cite{colding2012smooth}. The key ingredient is to show that if the convergence has multiplicity greater than $1$, then there exists a positive Jacobi field on $\Sigma$, which is a contradiction.

\begin{proof}
We argue as in \cite{choi1985space} that we only need to consider the case that $M$ is simply connected, and self-touching $\Sigma$ is two sided (note although in \cite{choi1985space} this argument concerns closed embedded surfaces, it can be adapted to self-touching surfaces). So if the convergence has multiplicity more than $1$, then $\Sigma_i$'s can be decomposed into several sheets of graphs on $\Sigma\setminus \mc S$. Since $\Sigma$ is two-sided, we can label the graphs by height, and let the highest sheet of $\Sigma_i$ can be written as the graph of $w_i^+$, and the lowest sheet of $\Sigma_i$ can be written as the graph of $w_i^-$, and let $w_i=w_i^+-w_i^-$. Fix a point $p$ not in $\mc S$, and let $u(x)=w(x)/w(p)$. Then $u(p)=1$ and $u>0$ on $\Sigma\setminus\mc S$. Moreover, although $w_i^-$ and $w_i^+$ do not satisfy a linear elliptic equation, but their difference does, hence $u_i$ satisfies a linear elliptic equation. Then Harnack inequality implies $C^\alpha$ bound for $u_i$'s and then standard elliptic theory gives $C^{2,\alpha}$ bound. Then by Arezela-Ascoli theorem, a subsequence (still denoted by $u_i$) converges uniformly in $C^2$ on compact subset of $\Sigma\setminus\mc S$ to a non-negative function $u$ on $\Sigma\setminus\mc S$ such that
\begin{equation}
Lu=0,u(p)=1.
\end{equation}
Next we show $u$ can extends smoothly across $\mc S$ to a solution of $Lu=0$. Again we follow the idea in \cite{white1987curvature} and \cite{colding2012smooth}. We only need to show $u$ is bounded around each singularity $y$, then by standard elliptic theory $u$ extends smoothly. Suppose $u_i$ satisfies the linearized equation
\[L(u_i)=\dv(a_i\cdot\nabla u_i)+b_i\cdot \nabla u_i+c_iu_i.\]
Then choose an exponential normal coordinates over $B_\eps(y)\subset\Sigma$ and a cylinder $N$ over $B_\eps(y)\cap \Sigma$, when $\eps$ is small, implicit function theorem gives a foliation of of graphs $v_t$ over $B_\eps(y)\cap\Sigma$ in $N$ so that
\[v_0(x)=0\mbox{ for all $x\in B_\eps(y)$, and $v_t(x)=t$ for all $x\in\partial B_\eps(y)$}.\]
By Harnack inequality, $t/C_i\leq v_t\leq C_it$ for some $C_i>0$. Since the right hand side of the linearized equation turns to $0$ as $i\to\infty$, $C_i$ actually has uniform bound. Then by maximum principle, $u_i$ is bounded on $B_\eps(y)$ by a multiple of its supremum on $B_\eps(y)\setminus B_{\eps/2}(y)$. Hence $u$ has a removable of singularity at $p$.

So there exists a non-negative solution $u$ of the linearize operator $Lu=0$. By $u(p)=1$, Harnack inequality implies $u$ is positive everywhere. Then by Lemma \ref{lemma-positive Jacobi field}, $\Sigma$ is stable. However, plugging in test function constant $1$ implies that no immersed CMC surface in positive Ricci three manifold can be stable, which is a contradiction. Then we conclude the convergence has multiplicity one.
\end{proof}

By \cite{allard1972first}, this theorem implies the smooth convergence across those density $1$ points. It remains to show the smooth convergence across those touching singularities which are not neck pinching singularities.

\begin{theorem}
The convergence is smooth besides those neck pinching singularities.
\end{theorem}

\begin{proof}
Let $y\in\mc S$ be a singularity with density greater than $1$, then by Theorem \ref{limit} locally $\Sigma$ is the union of two disks $D_1,D_2$. Then by definition of pinching points, we know that if $y$ is not a pinching point, locally $\Sigma_i=G^1_i\cup G^2_i$ be the union of two graphs over $D_1\setminus\{y\},D_2\setminus\{y\}$ respectively. Thus we only need to apply previous analysis to each graph $G^j_i$, and will get smoothly convergence across $y$.
\end{proof}

Finally we discuss the situation that the limit $\Sigma$ is an embedded minimal surface. Now multiplicity $2$ convergence may happen because the CMC surfaces can converge to $\Sigma$ from both side with different orientation. However, if the convergence if of multiplicity larger than $2$, then there are at least two graphs have the same orientation. Then we repeat the discussion to these graphs, will again get a positive Jacobi field, which is a contradiction. Therefore, the convergence at most has multiplicity $2$.

\vspace{5pt}
Combining all the ingredients in this section we conclude the main theorem Theorem \ref{T:Main theorem}.

\section{Touching Examples}\label{section:Touching}
In this section we give some examples of touching points of CMC surfaces in three manifolds. 

\begin{example}
{\bf Kissing itself:} Let us consider a sphere $\mb S_R$ with radius $R$ in $\mb R^3$. by quotient a $\mb Z^3$ action of $\mb R^3$, we get a torus $\mb T^3$, and the image of $\mb S_R$ in $\mb T^3$ is an embedded CMC surface when $R$ sufficiently small. Now we increase the radius of $\mb S_R$. Then for some specific $R_0$, in $\mb T^3$, $\mb S_{R_0}$ will kiss itself thus form a touching point. This is not a neck pinching point.

The touching set may be very large. For example, we can consider cylinder $\mc C_R$ with radius $R$ in $\mb R^3$. Using the same construction, we can see for some radius $R_0$, $\mc C_{R_0}$ kiss itself at a straight line, which is a $1$ dimensional curve.
\end{example}

\begin{example}
{\bf Unduloid neck pinching:} An unduloid is a one periodic CMC surface in $\mb R^3$. See \cite{hadzhilazova2007unduloids} for detailed description of unduloids.

The unduloid has two parameter $a,c$ to determine its shape, see \cite{hadzhilazova2007unduloids} Theorem 3.1. When we let $a\to 0,c\to 1/H$, we can see the family of unduloids will smoothly converge to the union of spheres besides the touchings of spheres. This is an example of neck pinching singularity. One can see that smooth convergence can not across these neck pinching points because the topology changes in the limit.

Of course, we can quotient $\mb R^3$ by some $\mb Z^3$ actions to make this example to be an example in a closed three manifold.
\end{example}

The reader may notice that these examples are not lie in a Ricci positive three manifold. It is not known that whether the touching behavior of CMC surfaces in positive Ricci three manifolds is simpler or not. So we suggest the following conjectures:

\begin{conjecture}
A self-touching CMC surface in positive Ricci three manifold can not carry infinitely many touching points.
\end{conjecture}

\begin{conjecture}
For a CMC surface in three manifold with $1$ dimensional touching set, the touching set must be a geodesic of the ambient space.
\end{conjecture}

Another important observation is that a touching point of a self-touching CMC surfaces can be generated by both kissing and neck pinching process. For example, in $\mb T^3$, a sphere kissing itself can be generated by both the first example and the second example. So a very natural question is whether any touching can be generated by both process? Some observation suggest the answer is probably NO:

\begin{example}
Consider two spheres in $\mb R^3$ kissing at a single point $p$. Then they can not be the limit of a sequence of CMC surfaces, hence $p$ can not be a neck pinching point of a sequence of embedded CMC surfaces.
\end{example}

Note in \cite{aleksandrov1962uniqueness} Alexandrov proved that any embedded CMC surface in $\mb R^3$ must be a standard sphere. Hence this proposition is obvious. However, in more general three manifolds it remains unknown. 

\begin{conjecture}
Suppose $M$ be a compact three manifold with positive Ricci curvature. Assume there exists $S_1\cup S_2$ be the union of two embedded CMC spheres in $M$ kissing at $p$. Then $p$ can not be a neck pinching point.
\end{conjecture}

\section{Eigenvalue Estimate of CMC Surfaces with small $\vert H\vert$}
In this section we want to discuss an application of our main theorem. We will give a lower bound of first eigenvalue of CMC surfaces in positive Ricci three manifold with small $\vert H\vert$.

The main idea is using a method by Choi and Wang in \cite{choi1983first}. In \cite{choi1983first} Choi and Wang used an identity by Reilly to estimate the first eigenvalue of minimal surface in three manifold. The main issue for generalizing their method to CMC surface is that we may not be able to control the term with mean curvature (in minimal surface case, this term vanishes). So we need more delicate estimate for each terms in Reilly's identity.

We first recall the proof in \cite{choi1983first}. They used a formula by Reilly. For $u$ be a smooth function defined on a bounded domain $\Omega$ we have
\begin{equation}
\begin{split}
\int_\Omega\left(\vert\nabla^2u\vert^2+\Ric(\nabla u,\nabla u)-(\Delta u)^2\right)
=\int_{\partial\Omega}(A((\nabla u)^\top,(\nabla u)^\top)-2u_n\Delta_{\partial\Omega}u+Hu_n^2),
\end{split}
\end{equation}
where $u_n$ is the normal derivative and $H$ is the mean curvature on $\partial\Omega$. Then they apply this formula to $\partial\Omega$ is minimal, where $u$ is the harmonic function solving
\[\Delta_\Omega u=0\mbox{ and }u\vert_{\partial\Omega}=f,\]
where $f$ an eigenfunction of the first eigenvalue on $\partial\Omega$ such that $\int_{\partial\Omega}f^2=1$. Then they could get first eigenvalue estimate for $\partial\Omega$, i.e. the minimal surfaces, in positive Ricci three manifolds which are simply connected. Later in \cite{choi1985space} Choi and Schoen used a covering argument to extend the estimate to all closed three manifold with positive Ricci curvature.

Let us naively follow their method to deal with CMC surfaces. Suppose now $\partial\Omega$ is a CMC surface with constant mean curvature $H$ and the first eigenvalue of $\partial\Omega$ is $\lambda$. We will get the following inequality (see \cite{colding2011course} page 244):
\begin{equation}
2\lambda\int_{\Omega}\vert\nabla u\vert^2\geq (\min\Ric)\int_\Omega\vert\nabla u\vert^2+\int_\Omega\vert\nabla^2 u\vert^2-\int_{\partial \Omega}A((\nabla u)^\top,(\nabla u)^\top)-H\int_{\partial\Omega}u_n^2.
\end{equation}
Since $A$ changes sign if we replace $\Omega$ by its complement, we can always assume $-\int_{\partial \Omega}A((\nabla u)^\top,(\nabla u)^\top)$ is non-negative and get
\begin{equation}\label{E:Reilly for CMC}
2\lambda\int_{\Omega}\vert\nabla u\vert^2\geq (\min\Ric)\int_\Omega\vert\nabla u\vert^2+\int_\Omega\vert\nabla^2 u\vert^2-H\int_{\partial\Omega}u_n^2.
\end{equation}
So our goal is to control $\int_\Omega\vert\nabla^2 u\vert^2-H\int_{\partial\Omega}u_n^2$.
\subsection{Trace Theorem}
In this subsection, we will transform the problem of controlling $\int_\Omega\vert\nabla^2 u\vert^2-H\int_{\partial\Omega}u_n^2$ to the problem of getting uniform tubular neighbourhood of CMC surfaces.  We need a trace theorem in three manifold. The idea of the proof is based on the proof in \cite{evans2010partial}.

\begin{theorem}\label{T:Trace on manifold}
Suppose $\Sigma=\partial\Omega$ is an embedded surface in a three manifold $M$. Suppose there is a constant $\delta$ such that $\exp_x(t\mathbf{n}):\Sigma\times[-\delta,\delta]\to M$ is a diffeomorphism from $\Sigma\times[-\delta,\delta]$ to its image, and there is a constant $A_0$ such that $\vert A\vert\leq A_0$ on $\Sigma$, then there is a constant $C$ only depending on $M,A_0$ and $\delta$ such that
\begin{equation}
\int_{\partial\Omega}(u_n)^2\leq C \int_{\Omega}(\vert\nabla u\vert^2+	\vert\nabla^2u\vert^2).
\end{equation}
\end{theorem}

\begin{proof}
Note $(u_n)^2\leq\vert\nabla u\vert^2$, so we only need to prove a standard trace theorem
\[\int_{\partial\Omega}f^2\leq C \int_{\Omega}(f^2+	\vert\nabla f\vert^2).
\]

By the conditions, we can pull back the metric of $M$ to $\Sigma\times[-\delta,\delta]$, and by the uniform curvature bound, the pull back metric is uniformly closed to the standard production metric. In particular, we only need to prove the trace theorem on $\Sigma\times[-\delta,\delta]$ with product metric. Let us choose a cut-off function $\zeta$ such that $\zeta=1$ on $\Sigma\times[-\delta/2,\delta/2]$ and supported on $\Sigma\times[-\delta,\delta]$. Moreover we may assume its gradient is bounded by $C/\delta$ for some constant $C$. Then 

\begin{equation}
\begin{split}
\int_{\partial\Omega}f^2dx'&=\int_\Sigma f^2\zeta dx'=-\int_{\Sigma\times[-\delta,0]}(f^2\zeta)_{x_n}dx\\
&=-\int_{\Sigma\times[-\delta,0]}\vert f\vert^2\zeta_{x_n}+2f f_{x_n}\zeta dx\\
&\leq C\int_{\Sigma\times[-\delta,0]}\vert f\vert^2+\vert\nabla f\vert^2 dx.
\end{split}
\end{equation}

Here $x_n$ is the normal direction (i.e. the direction on $[-\delta,\delta]$), and $dx'$ is the measure on $\Sigma$, $dx$ is the measure of the production metric. In the last inequality we use Young's inequality. Then translating this back to $M$ we get the desired trace theorem.
\end{proof}

If $H$ is sufficiently close to $0$, then we apply this trace theorem into inequality (\ref{E:Reilly for CMC}) to get the eigenvalue lower bound:
\begin{equation}
\lambda\geq\frac{\min\Ric-HC}{2}
\end{equation}
where $C$ is a constant depending on $M,A_0,\delta$ in trace theorem.
\subsection{Uniform Bound for CMC Surfaces with $H$ close to $0$}
It remains to prove the pointwise curvature bound and the existence of $\delta$ in Theorem \ref{T:Trace on manifold}. We will use the compactness theorem to get these bounds for CMC surfaces with $H$ small.

\begin{theorem}
Suppose there is no embedded minimal surface in $M$ which is the multiplicity $2$ limit of a sequence of CMC surfaces. There exists $H_0>0$ such that for embedded CMC surface $\Sigma$ with mean curvature $\vert H\vert\leq H_0$ area less than $V$ and genus less than $G$, its curvature $\vert A\vert\leq C(H_0,V,G)$
\end{theorem}

\begin{proof}
We argue by contradiction. Suppose such $H_0$ does not exist. Then we can find a family of CMC surfaces $\Sigma_i$, with mean curvature $H_i\to 0$ such that a point $p_i$ on $\Sigma_i$ has curvature $\vert A(p_i)\vert\to\infty$ as $\i \to \infty$. By compactness of $M$ we may assume $p_i\to p$ for a point $p\in M$. Now by main Theorem \ref{T:Main theorem}, $\Sigma_i$ converges to a minimal surface $\Sigma$. Since $\Sigma$ is minimal, by maximum principle there is no touching point. So the convergence is everywhere smooth. However $\vert A(p)\vert$ is finite since $\Sigma$ is an smoothly embedded surface, which is a contradiction. Thus $H_0$ exists. 
\end{proof}

\begin{theorem}
Suppose there is no embedded minimal surface in $M$ which is the multiplicity $2$ limit of a sequence of CMC surfaces. There exists $H_0>0$ and $\delta_0>0$ such that for embedded CMC surface $\Sigma$ with mean curvature $\vert H\vert\leq H_0$ area less than $V$ and genus less than $G$, $\exp_x(t\mathbf{n}):\Sigma\times[-\delta,\delta]\to M$ is a diffeomorphism.
\end{theorem}

\begin{proof}
We argue by contradiction. Suppose such $H_0,\delta_0$ does not exist. Then we can find a family of CMC surfaces $\Sigma_i$, with mean curvature $H_i\to 0$ and $\delta_i\to 0$ such that there is a point $p\in M$ such that $p=\exp_{x_i^j}(t_i^j\mathbf{n}),j=1,2$ for $x_i^j\in\Sigma_i$ and $\vert t_i^j\in[-\delta_i,\delta_i]$. Since we already get uniform curvature bound for $\Sigma_i$'s, we know $\dist_\Sigma(x_i^1,x_i^2)\geq d$ for some constant $d$ when $i$ large enough.

Again, a subsequence of $\Sigma_i$ smoothly converge to a smooth embedded minimal surface $\Sigma$. passing to subsequence we can find two points $x^1,x^2\in\Sigma$, with intrinsic distance $\dist_\Sigma(x^1,x^2)\geq d$ but extrinsic distance $\dist_{M}(x^1,x^2)=0$. This is a contradiction by maximum principle of minimal surface.
\end{proof}

Combining all the ingredients in this section we get the following lower bound for the first eigenvalue of CMC surfaces:

\begin{theorem}[Theorem \ref{T:Application eigenvalue}]
Let $M$ be a three manifold with positive Ricci curvature. Suppose there is no embedded minimal surface in $M$ which is the multiplicity $2$ limit of a sequence of CMC surfaces. Then for any embedded CMC surface with area bound $V$, genus bound $G$ and mean curvature bound $\vert H\vert\leq H_0$, we have the first eigenvalue lower bound:

\begin{equation}
\lambda\geq \frac{\min\Ric -HC}{2},
\end{equation}
where $C$ is a constant depending on $M,V,G,H_0$.
\end{theorem}

\begin{remark}
An interesting question is: can we directly get the first eigenvalue lower bound for CMC surfaces? If we can, then we can prove the compactness theorem for CMC surfaces without area bound.
\end{remark}
\section*{Appendix: Difference of Two Surfaces in Three Manifold}
In this appendix, we will present some computations of the difference of two surfaces in three manifold. These kind of computations have already appeared in Kapouleaus \cite{kapouleas1990complete} and Colding-Minicozzi \cite{colding2011course} in three dimensional Euclidean space.

Let $\Sigma_1,\Sigma_2$ be two surface in three manifold $M$, and let $H_1,H_2$ be their mean curvature respectively. Moreover, we assume $\Sigma_2$ can be viewed as a graph over $\Sigma_1$, i.e.

\[\Sigma_2=\{\exp_x(\varphi \bn):x\in\Sigma_1\},\]

where $\varphi$ is a $C^2$ function on $\Sigma_1$.

\begin{theorem}
Suppose $\Vert \varphi\Vert_{C^2}$ small enough, then $\varphi$ satisfies a second order elliptic equation.
\end{theorem}

\begin{proof}
Since this assertion is a local assertion, we only need to check this in a small neighbourhood $U$ of $p\in\Sigma_1$. Let us choose the Fermi coordinate $x_1,x_2,x_3$ in $U$ (so we can view $U$ as an open subset of $\mb R^3$ with non-Euclidean metric), such that 

\[\Sigma_1=\{(x_1,x_2,x_3):x_3=0\}.\]

Moreover, the metric $g$ under this coordinate satisfies $g_{i3}=0,i=1,2$, and $(0,0,1)$ is the unit normal vector at each points in $\Sigma_1$. We will use $\partial_1,\partial_2,\partial_3$ to denote the vector fields defined on $M$ with respect to the differential under this coordinate.

$\Sigma_2$ is a graph

\[\Sigma_2=\{(x_1,x_2,x_3):x_3=\varphi(x_1,x_2)\}.\]

From now on we will use tilde over quantity to denote the quantity of $\Sigma_2$. We use $x_1,x_2$ to parametrize $\Sigma_2$, then we have
\begin{equation}
\tilde\partial_{i}=\tilde \partial_{x_i}=\partial_i+\varphi_i\partial_3,i=1,2.
\end{equation}

Then the metric on $\Sigma_2$ satisfies
\begin{equation}
\tilde g_{ij}=g_{ij}+g_{i3}\varphi_j+g_{j3}\varphi_i+g_{33}\varphi_i\varphi_j.
\end{equation}

Now we compute the unit normal vector fields on $\Sigma_2$. We observe that $\Sigma_2$ can be viewed as the $0$-level set of the function $\varphi(x_1,x_2)-x_3$. So we can find a normal vector field $\mathbf{m}$ on $\Sigma_2$:

\begin{equation}
\mathbf{m}=-\nabla^M(\varphi(x_1,x_2)-x_3)=g^{ij}(\varphi_k\delta_{i}^k-\delta_{3i})\partial_j.
\end{equation}

Note
\[\langle \mathbf{m},\mathbf{m}\rangle=g^{ij}(\varphi_k\delta_{i}^k-\delta_{3i}) g^{pq}(\varphi_k\delta_{p}^k-\delta_{3p})g_{qj}=g^{ij}\varphi_i\varphi_j-2g^{3k}\varphi_k+g^{33}.\]

So 
\begin{equation}
\bn=-(g^{pq}\varphi_p\varphi_q-2g^{3l}\varphi_l+g^{33})^{-1/2}g^{ij}(\varphi_k\delta_{i}^k-\delta_{3i})\partial_j.
\end{equation}

Now let us calculate the mean curvatures. From now on we will slightly abuse the notation when we use Einstein notation. When we use $i,j$ in the summation we will assume they are in $\{1,2\}$. On $\Sigma_1$, $\bn=\partial_3$, so we have 

\begin{equation}
H_1=g^{ij}\langle \nabla_{\partial_i}\partial_j,\partial_3\rangle=g^{ij}\Gamma_{ij}^kg_{k3}.
\end{equation}

On $\Sigma_2$, first we note the covariant derivative is
\begin{equation}
\begin{split}
\nabla_{\tilde{\partial_i}}\tilde{\partial_j}&=\nabla_{\partial_i+\varphi_i\partial_3}(\partial_j+\varphi_j\partial_3)\\
&=\nabla_{\partial_i}\partial_j+\varphi_i\nabla_{\partial_3}\partial_j+\varphi_{ij}\partial_3+\varphi_i\varphi_j\nabla_{\partial_3}\partial_3\\
&=\nabla_{\partial_i}\partial_j+\varphi_i\nabla_{\partial_3}\partial_j+\varphi_{ij}\partial_3.
\end{split}
\end{equation}

Here we note that since $\partial_3$ is the direction of the geodesic starting from $\Sigma_1$, hence $\nabla_{\partial_3}\partial_3=0$. Then the mean curvature of $\Sigma_2$ is

\begin{equation}
\begin{split}
H_2&=\tilde{g^{ij}}\langle \nabla_{\tilde{\partial_i}}\tilde{\partial_j},\bn\rangle\\
&=\tilde{g^{ij}}\langle  \nabla_{\partial_i}\partial_j+\varphi_i\nabla_{\partial_3}\partial_j+\varphi_{ij}\partial_3 ,-(g^{pq}\varphi_p\varphi_q-2g^{3l}\varphi_l+g^{33})^{-1/2}g^{rs}(\varphi_k\delta_{r}^k-\delta_{3r})\partial_s\rangle\\
&=-\tilde{g^{ij}}(g^{pq}\varphi_p\varphi_q-2g^{3l}\varphi_l+g^{33})^{-1/2}g^{rs}(\varphi_k\delta_{r}^k-\delta_{3r})(\Gamma_{ij}^mg_{ms}+\varphi_i\Gamma_{3j}^mg_{ms}+\varphi_{ij}g_{3s})
\end{split}
\end{equation}

In conclusion, $H_2$ is a function of $\varphi,\nabla\varphi,\nabla^2\varphi$ and the coordinate in ambient manifold. Now let us define a function $H(x_1,x_2,x_3,v_i,w_{ij})$ where $H(x,y,\varphi,\nabla\varphi,\nabla^2\varphi)=H_2$ as above, where $(x_1,x_2,x_3)$ is the local coordinate. Also note that $H(x,y,0,0,0,0)=H_1$. Then we have
\begin{equation}
\begin{split}
H_2-H_1=&\varphi\int_{0}^{1}\frac{\partial H(x_1,x_2,t\varphi,t\nabla\varphi,t\nabla^2\varphi^2)}{\partial x_3}dt+\varphi_i\int_{0}^{1}\frac{\partial H(x_1,x_2,t\varphi,t\nabla\varphi,t\nabla^2\varphi^2)}{\partial v_i}dt\\
&+\varphi_{jk}\int_{0}^{1}\frac{\partial H(x_1,x_2,t\varphi,t\nabla\varphi,t\nabla^2\varphi^2)}{\partial w_{jk}}dt.
\end{split}
\end{equation}

Let the coefficients of $\varphi,\nabla\varphi,\nabla^2\varphi$ on the right hand side of the above identity be a function depending on $\varphi,\nabla\varphi,\nabla^2\varphi$. Then let $\Vert\varphi\Vert_{C^2}$ goes to $0$ we can see the right hand side terms will just be $Lu$ by the second variational formula. Thus we have

\begin{equation}
Lu-(H_2-H_1)=\dv(a\nabla\varphi)+b\cdot\nabla \varphi+c\varphi.
\end{equation}
Where $a,b,c$ turns to $0$ as $\Vert\varphi\Vert_{C^2}$ goes to $0$.

\end{proof}

\bibliography{bibfile}
\bibliographystyle{alpha}
\end{document}